\documentclass[12pt, a4paper]{article}

\usepackage{inputenc} 
\usepackage[english]{babel}
\usepackage{dsfont}

\usepackage[totalheight=24 true cm, totalwidth=17 true cm]{geometry}
\usepackage{enumitem}

\usepackage{bbding}
\usepackage{amsmath}
\usepackage{amsthm}
\usepackage{amssymb}
\usepackage{mathrsfs}
\usepackage{mathtools}
\usepackage{ stmaryrd }
\usepackage{url}
\usepackage{wrapfig}
\usepackage{starfont}
\usepackage{pifont}
\usepackage{eurosym}

\usepackage[maxbibnames=99]{biblatex}
\usepackage{imakeidx}
\makeindex[]

\usepackage{multicol}
\usepackage{tikz,pgf}
\usepackage{tikz-cd}
\usepackage{wasysym}
\usepackage{graphicx}
\usepackage{fancyhdr}
\usepackage{verbatim}

\usepackage[all,dvips,knot,web,arc,curve,color,frame]{xy}
\usepackage[all,knot,arc,color,web]{xy}
\xyoption{arc}
\xyoption{web}
\xyoption{curve}

\numberwithin{equation}{section}

\theoremstyle{plain}
\newtheorem{theorem}{Theorem}[section]
\newtheorem{proposition}[theorem]{Proposition}
\newtheorem{corollary}[theorem]{Corollary}

\newtheorem{definition}[theorem]{Definition}

\theoremstyle{definition}

\newtheorem{remark}[theorem]{Remark}

\newtheorem{notation}[theorem]{Notation}
\newtheorem{example}[theorem]{Example}

\usepackage[bookmarksnumbered=true]{hyperref} 
\hypersetup{
     colorlinks = true,
     linkcolor = blue,
     anchorcolor = blue,
     citecolor = teal,
     filecolor = blue,
     urlcolor = blue
     }

\DeclareMathOperator{\rank}{rk}
\DeclareMathOperator{\jac}{J}
\DeclareMathOperator{\Sect}{Sect}
\DeclareMathOperator{\im}{im}
\DeclareMathOperator{\supp}{supp}
\DeclareMathOperator{\sign}{sign}
\DeclareMathOperator{\sspan}{span}
\DeclareMathOperator{\triv}{triv}
\DeclareMathOperator{\CC}{\mathbb{C}}
\DeclareMathOperator{\FF}{\mathbb{F}}
\DeclareMathOperator{\QQ}{\mathbb{Q}}
\DeclareMathOperator{\RR}{\mathbb{R}}

\usepackage[maxbibnames=99]{biblatex}
\title{On the Symmetric Tensors Completion Problem}
\author{Shokhzod Kurokboev}
\date{}

\begin{document}
\maketitle

\begin{abstract}
The aim of this elaborate is presenting the classical symmetric tensors completion problem to an audience of graduate students. As main studying tool, we will introduce the theory of hypergraph rigidity which naturally mirrors the problem itself. This has already appeared in literature: in 2023, Cruicksand, Mohammadi, Nixon, and Tanigawa introduced organically the theory of hypergraph rigidity, as a generalization of the usual Euclidean rigidity, and showed how it can be used to mirror several different theoretical problems, including tensor completion.
 \end{abstract}
 
{\hypersetup{linkcolor=black}
{\tableofcontents}}


\section{Introduction}

This note presents the symmetric tensor completion problem which has application in many areas including phylogenetics, quantum information and signal processing. In the first section, we introduce the symmetric tensor completion problem after recalling the basics on tensors. In the next section, we give the basics of Euclidean rigidity and generalize it to hypergraph rigidity. This section also introduces the infinitesimal $g$-rigidity and provides its relation with certain properties of secant varieties such that identifiability and defectivity. The last section gives sufficient conditions for local and global $g$-rigidity. We mainly discuss to check for $t$-tangentially weak defectiveness which is one of the sufficient conditions for being globally $g$-rigid in an algebraic way.  

\section{Symmetric Tensors}
\subsection{Basics on Tensors}

\begin{definition}[Tensor]
A $n$-th rank tensor in $m$-dimensional space is a mathematical object that has $n$ indices and $m$ components and obeys certain transformation rules. Each index of a tensor ranges over the number of dimensions of space. 
\end{definition}

\begin{remark}[Representation of Tensors]
\end{remark}

Let $V$ be a vector space of dimension $n$ over $\mathbb{C}$ and $V^{\otimes k}$ be the $k$-fold tensor product of $V$. Let fix a basis of $V$, and assume that each $T\in V^{\otimes k}$ is represented by a $k$-dimensional array over $\mathbb{C}$.

\begin{definition} 
A tensor $T\in V^{\otimes k}$ is called symmetric if for any permutation $\sigma$ on $[k]$ we have $T_{i_{1},...,i_{k}}=T_{\sigma(i_{1}),...,\sigma (i_{k})}$. The set of symmetric tensors in $V^{\otimes k}$ is denoted by $S^k(V)$.
\end{definition}
It is well-known fact that any symmetric tensor can be written as follows:
\begin{center}
 $T=\sum_{i=1}^{d} x^{\otimes k}_{i}=\sum_{i=1}^{d} x_{i} \otimes x_{i}\otimes...\otimes x_{i} $    (1)   
\end{center}
for some vectors $x_{1},...,x_{d}\in V$
\begin{definition}
 The smallest possible $d$ for which the symmetric tensor $T$ can be written in the form of equation (1) is called the symmetric rank of $T$. 
\end{definition}

\subsection{The Symmetric Tensors Completion Problem}

Let $T$ be a partially-filled tensor of order $k$ and size $n$ and in the symmetric tensor completion problem we are asked to fill the remaining entries to obtain a symmetric tensor of symmetric rank at most $d$. Let denote the set of multisets of $k$ elements of a finite set $X$ with $\binom{[n]}{k}$. Hence, we can encode the underlying combinatorics of each instance of the completion problem using a $k$-uniform hypergraph $([n],E)$. 

We can also write equation (1) by using an algebraic relation among points in $\mathbb{C}^d$. For this, let $p$ be the $d\times n$ matrix with the $i$-th row equal to $x_{i}$. The $e$-th entry of equation (1)

\section{$\mathbf{g}$-Rigidity of Hypergraphs}

\subsection{Euclidean rigidity}
Let $G=(V,E)$ be the graph whose vertex set is $V$ and whose edge set is $E$.
\begin{definition} (Framework, Generic Framework)
 A framework $(G,p)$ is an ordered pair containing of a graph $G$ and a point-configuration $p:V\longrightarrow \mathbb{R}^d$. In particular, a framework is told generic if there is no algebraic relation over $\QQ$, among the coordinates of the points in $p(V) \subseteq \RR^d$
\end{definition}

\begin{definition}(Rigid and Flexible Frameworks)
 A framework $(G,p)$ is rigid if every edge-length preserving continuous motion of the vertices arises from isometries of $\mathbb{R}^d$. Otherwise, it is called flexible. 
\end{definition}

\begin{example}
  Complete graphs are always rigid. This is because we know that the framework $(G,q)$ can be obtained from the framework $(G,p)$ by an isometry of $\mathbb{R}^d$ if and only if frameworks $(G,p)$ and $(G,q)$ are congruent which means that $||p(v_{i})-p(v_{j})||=||q(v_{i})-q(v_{j})||$ for any $v_{i},v_{j}\in V$. In the complete graph, any edge-length preserving continuous motion of the vertices is already congruent to the previous one. 
  
\end{example}
\begin{example}
 Any disconnected graph is not rigid (in any dimension). The graph $G$ is disconnected, hence there exist two vertices $v_{1}$ and $v_{2}$ such that no path between them. It is clear that, we can continuously move vertices in a such way that, edge-lengths are preserved but the length between the vertices $v_{1}$ and $v_{2}$ is not. 
\end{example}

\begin{example}(Rigidity depends on the dimension).
Let $G$ be a square with one diagonal, more formally $V=(v_{1},v_{2},v_{3},v_{4})$ and $E=(v_{1}v_{2},v_{2}v_{3},v_{3}v_{4},v_{4}v_{1},v_{1}v_{3})$. It is rigid in $\mathbb{R}^2$, because any edge-length preserving continuous motion of the vertices is congruent to the previous one. However, it is not rigid in $\mathbb{R}^3$. Because we can move the triangle $v_{1}v_{2}v_{3}$ by fixing the triangle $v_{1}v_{3}v_{4}$ in $\mathbb{R}^3$. It is edge-length preserved motion, but the length between the vertices $v_{2}$ and $v_{4}$ is not preserved.

\end{example}

\begin{example}(Rigidity depends on the point-configuration). Let $G_{1}$ be a square in $\mathbb{R}^2$ with $V=(v_{1},v_{2},v_{3},v_{4})$ and $E=(v_{1}v_{2},v_{2}v_{3},v_{3}v_{4},v_{4}v_{1})$. In this example, $p_{1}:v_{1}\longmapsto(0,0), v_{2}\longmapsto(0,1), v_{3}\longmapsto(1,1), v_{4}\longmapsto(1,0)$. The graph $G_{1}$ is not rigid, because we can get an unit rhombus by moving the vertices continuously. It is an edge-length preserving motion, but the length between the vertices $v_{1}$ and $v_{3}$ is not preserved. Let $G_{2}$ be a triangle in $\mathbb{R}^2$, which $p_{2}:v_{1}\longmapsto (0,0), v_{2}\longmapsto (1,1)$, $v_{3}\longmapsto (1,1)$, $v_{4}\longmapsto (1,0)$. $G_{2}$ is a complete graph and hence it is rigid in $\mathbb{R}^2$.
    
\end{example}

\begin{definition}(Measurement Map)
 The measurement map $f_{G}:(\mathbb{R}^d)^{V}\longrightarrow\mathbb{R}^{E}$ is defined by putting 
$$f_{G}(p)=(...,|| p(i)-p(j)||^{2},...)_{ij\in E},$$
where $|| \bullet ||$ denotes the usual Euclidean metric.
\end{definition}

\begin{definition}[Infinitesimal Motion, Trivial Motion, and Infinitesimally Rigid Framework]
An infinitesimal motion of the framework $(G,p)$ is a map $u \colon V\to \mathbb{R}^d$ such that $(p(i)-p(j))(u(i)-u(j))=0$ for all $ij\in E$.

An infinitesimal motion is termed trivial if $u=Ap(i)+b$ for all $i\in V$, where A is a skew-symmetric matrix and $b$ is a vector in $\mathbb{R}^d$.

A framework $(G,p)$ is infinitesimally rigid if every infinitesimal motion of $(G,p)$ is trivial. 
\end{definition}

\begin{example}
 Let $G=(\{1,2,3\}, \{12,13,23\})$ is a complete graph (triangle) in $\mathbb{R}^2$. Then infinitesimal motion of $(G,p)$, $u:V\longrightarrow\mathbb{R}^2$ satisfies
 \begin{center}
  $$\begin{cases}
   (p_{1}-p_{2})(u_{1}-u_{2})=0\\
   (p_{1}-p_{3})(u_{1}-u_{3})=0\\
   (p_{2}-p_{3})(u_{2}-u_{3})=0
  \end{cases} \hbox{or equivalently,} \quad \begin{cases}
  (p_{1}-p_{2})u_{1}+(p_{2}-p_{1})u_{2}+0u_{3}=0\\
  (p_{1}-p_{3})u_{1}+0u_{2}+(p_{3}-p_{1})u_{3}=0\\
  0u_{1}+(p_{2}-p_{3})u_{2}+(p_{3}-p_{2})u_{3}=0
  \end{cases}$$
 \end{center}
 Hence, its rigidity matrix equals:
 \begin{center}
 $$\begin{pmatrix}
  p_{1}-p_{2}&p_{2}-p_{1}&0\\
  p_{1}-p_{3}&0&p_{3}-p_{1}\\
  0&p_{2}-p_{3}&p_{3}-p_{2}
 \end{pmatrix}$$
  \end{center}
\end{example}

\begin{definition}[Globally and Locally Rigid Framework]
    Let $(G,p)$ be a framework in $\RR^d$ and let $f$ be the corresponding measurement map. Then $(G,p)$ is said to be globally rigid if every $q\in f^{-1}_{G}(f_{G}(p))$ satisfies $t\cdot q=p$ for an isometry $t$ of $\mathbb{R}^d$.
   
\end{definition}

\begin{example}
 Complete graphs are always globally rigid.
\end{example}

\begin{example}

Let $G$ be a square in $\mathbb{R}^2$ with $V=(v_{1},v_{2},v_{3},v_{4})$ and $E=(v_{1}v_{2},v_{2}v_{3},v_{3}v_{4},v_{4}v_{1})$. In this example, $p_{1}:v_{1}\longmapsto(0,0), v_{2}\longmapsto(0,1), v_{3}\longmapsto(1,1), v_{4}\longmapsto(1,0)$. It is rigid in $\mathbb{R}^2$ because of the same reason we have seen in Example 3.5. But, it is not globally rigid in $\mathbb{R}^2$. Indeed, if we take a point configuration $q$ which sends $v_{1}\longmapsto(1,1), v_{2}\longmapsto(0,1), v_{3}\longmapsto(1,1), v_{4}\longmapsto(1,0)$, then $q\in f^{-1}_{G}(f_{G}(p))$. On the other hand, $(G,p)$ and $(G,q)$ are not congruent.

\end{example}

\subsection{A Generalization of Euclidean Rigidity}

In rigidity theory all notions can be extended by replacing the Euclidean distance function with a general smooth function. In this case, we will use hypergraphs instead of graphs to capture the underlying combinatorics. 

\begin{notation}
In the following, $\mathbb{F}$ is either the field of real numbers $\mathbb{R}$ or the field of complex numbers $\mathbb{C}$.
\end{notation}

\begin{definition}[Hypergraph and Hyperframework]
 A hypergraph is an ordered pair of $(V,E)$, where $V$ is a set of vertices and $E$ is a set of pairs of subsets of $V$. A point configuration is a map $p:V\longrightarrow\mathbb{F}^d$. An ordered pair $(G,p)$, containing of a hypergraph $G$ and point-configuration $p$ is called hyperframework in $\mathbb{F}^d$.   
\end{definition}

\begin{example}
 Let $d=2$ and $G=(V,E)$ is a hypergraph, where $V=\{a,b,c,d\}$ and $E=\{aaa,aab,aac,abc\}$. A point-configuration $p:V\longrightarrow \mathbb{F}^2$ sends the vertices $a\longmapsto(x_{a_{1}},x_{a_{2}})$,
 $b\longmapsto(x_{b_{1}},x_{b_{2}})$,
 $c\longmapsto(x_{c_{1}},x_{c_{2}})$ and $d\longmapsto(x_{d_{1}},x_{d_{2}})$.
\end{example}

We will now adapt the definitions of the tools we introduced for the usual rigidity to the generalized context of hypergraph rigidity.

\begin{definition}[$g$-measurement map]
 Let $G = (V,E)$ be a $k$-uniform hypergraph, let $(G,p)$ be a corresponding hyperframework, and let $g \colon (\FF^d)^k \to \FF$ be a polynomial map. The $g$-measurement map of $G$ is defined as the polynomial map:
 \begin{align*}
     f_{g,G} \colon (\mathbb{F}^d)^V &\to \mathbb{F}^E \\
     \hfill p(V) &\mapsto (g(p(v_1), \dots, p(v_k)))_{\{v_1, \dots, v_k\} \in E} 
 \end{align*}

\end{definition}

\begin{remark}
\begin{itemize}
    \item With a minor abuse of notation, we will write $f_{g,G}(p):=(g(p(v_1),p(v_2),...,p(v_k)):e=\{v_1,v_2,...,v_k\}\in E)$.
    \item The good definition of the $g$-measurement map $f_{g,G}$ requires the polynomial map $g$ to be either symmetric or anti-symmetric with respect to the ordering of the points.
    \item For the classical rigidity theory, $g$ is the Euclidean distance.
\end{itemize}

\end{remark}
 
Suppose the general affine group $\hbox{Aff} (d,\mathbb{F})$ acts on $\mathbb{F}^d$ by $\gamma \cdot x=Ax+t$ for $x\in \mathbb{F}^d$ and each pair $\gamma=(A,t)$, where $A\in$GL$(d,\mathbb{F})$ and $t\in \mathbb{F}^d$. Similarly, the action of $\hbox{Aff} (d,\mathbb{F})$ on $(\mathbb{F}^d)^V$ is defined on the following way: $(\gamma\cdot p)(v)=Ap(v)+t$
for any $\gamma=(A,t)\in \hbox{Aff} (d,\mathbb{F})$ and $p\in (\mathbb{F}^d)^V$.  Then the induced action on a polynomial map $g:(\mathbb{F}^d)^k\longrightarrow\mathbb{F}$ is given by $\gamma \cdot g(x_1,...,x_k)=g(\gamma^{-1}\cdot x_1,...,\gamma^{-1}\cdot x_k)$ for $x_1,...,x_k\in\mathbb{F}^d$ and $\gamma\in \hbox{Aff} (d,\mathbb{F})$. $\gamma$ is called stabilizes $g$ if $g$ is invariant respect to the action $\gamma$ and the set of all pairs $(A,t)$ which stabilize $g$ called the stabilizer $\Gamma_g$ of $g$.

\begin{definition}(Globally and Locally $g$-rigid Hyper-framework)
The hyper-framework $(G,p)$ is called globally $g$-rigid if for any $q\in f^{-1}_{g,G}(f_{g,G}(p))$ there exists $\gamma\in \Gamma_{g}$ such that $q=\gamma \cdot p$. The hyper-framework $(G,p)$ is called locally $g$-rigid if there exists an open neighbourhood $N$ of $p$ in $(\mathbb{F}^d)^V$ respect to the Euclidean topology such that for any
$q\in f^{-1}_{g,G}(f_{g,G}(p))\cap N$ there is  $\gamma\in \Gamma_{g}$ such that $q=\gamma \cdot p$.
\end{definition}
  
Let $(G,p)$ be a $k$-uniform hyper-framework and let $g \colon (\FF^d)^k \to \FF$ be a polynomial map.. The stabilizer $\Gamma_{g}$ of $g$ is a Lie group, let denote Lie algebra of $\Gamma_{g}$ by $g$. $\Gamma_{g}$ acts on $(\mathbb{F}^d)^V$ on the following way: $(\gamma \cdot p)(v)=\gamma \cdot (p(v))$ for $v\in V$ and it induces a map from $g\times (\mathbb{F}^d)^V \longrightarrow (\mathbb{F}^d)^V$, sends the tuple $(\dot{\gamma},p)$ to $\dot{\gamma}\cdot p$ whose image lies in the right kernel of $Jf_{g,G}(p)$. Then $\dot{p}\in {\mathbb{F}^d}^V$ which is defined by $\dot{p}=\dot{\gamma}\cdot p$ is called a trivial infinitesimal $g$-motion of the hyper-framework $(G,p)$. The space of trivial infinitesimal motions of $(G,p)$ is defined as follows: triv$_{g}(p)=\{\dot{\gamma}\cdot p : \dot{\gamma}\in g\}$.

\begin{definition}
 Let $(G,p)$ be a $k$ uniform hyper-framework and let $g \colon (\FF^d)^k \to \FF$ be a polynomial map. The hyperframework $(G,p) $ is infinitesimally $g$-rigid if the dimension of $\ker Jf_{g,G}(p)$ is equal to the dimension of the space of trivial infinitesimal $g$-motions. 
\end{definition}

\begin{proposition}
 Let $(G,p)$ be a $k$ uniform hyper-framework and let $g \colon (\FF^d)^k \to \FF$ be a polynomial map.
 \begin{itemize}
   \item   If $(G,p)$ is infinitesimally $g$-rigid, then it is locally $g$-rigid.
   \item If $p$ is generic, then $(G,p)$ is infinitesimally $g$-rigid if and only if it is locally $g$-rigid.
 \end{itemize}
 
\end{proposition}
Let us denote the dimension of the $\Gamma_{g}$ by $d_{\Gamma_{g}}$ and $n_{\Gamma_{g}}$ is the smallest integer such that $\dim \triv_{g}(p)=d_{\Gamma_{g}}$ for some $p\in (\mathbb{F}^d)^n$.

\begin{proposition}
 Let $(G,p)$ be a $k$ uniform hyperframework whose hypergraph $G=(V,E)$ is such that $ | V | \ge n_{\Gamma_{g}}$. Then let $g:(\mathbb{F}^d)^k\longrightarrow\mathbb {F}$ be a polynomial map. Then $\rank \jac f_{g,G}(p)\le d |V| -d_{\Gamma_{g}}$, and equality holds if and only if $(G,p)$ is locally $g$-rigid.
\end{proposition}

\subsection{Infinitesimal $g$-rigidity and its Correlation with Certain Properties of Secant Varieties}

\begin{definition}
 Let $V$ be an affine variety in $\mathbb{C}^m$. Its $r$-secant variety is defined as $\Sect_{r}(V)= \overline {\bigcup_{x_1,...,x_r\in V} \langle x_1,...,x_r\rangle}$.
\end{definition}


The dimension of the $r$-th secant of an affine variety $V$ is at most $\min \{r \dim V,m\}$. The latter number is called the expected dimension, if $\dim \Sect _{r}(V)$ is smaller than the expected dimension then $V$ is called $r$-defective. Suppose that $g$ is the sum of $t$ copies of $h$. If $\overline{\im f_{h,G}}$ is not $t$-defective, then checking local $g$-rigidity is reduced to checking local $h$-rigidity. 

\begin{definition}
 Let $V$ be a variety in $\mathbb{C}^m$, a generic point $y\in \Sect _{r}(V)$ can be written as $y=\sum_{i=1}^{r} y_{i}$ for some $y_{i}\in V$. If $y_{1},...,y_{r}$ are uniquely determined up to permutations of indices and scaling of each $y_{i}$, then the variety $V$ is called $r$-identifiable.
\end{definition}

In the next section, we will see that identifiability is a powerful tool for checking global $g$-rigidity. 

\section{A Combinatorial Model for the Symmetric Tensors Completion Problem}
\subsection{Sufficient Conditions for Local $g$-rigidity}

In this section, we will give a sufficient condition for the local rigidity of hypergraphs $G$ which is called a packing-type condition.

Suppose that $g$ is the sum of $t$ copies of $h$. For a hyperedge $e\in \binom{[n]}{k}$ and $u\in e$, $e-u$ (resp., $e+u$) denotes the multiset obtained from $e$ by reducing (resp., increasing) the multiplicity of $u$ by one. We denote by $\supp(e)$ the set obtained from $e$ by ignoring the multiplicity of each element. Let $G=(V,E)$ be a $k$-uniform hypergraph with vertex set $V$ and hyperedge set $E$. For $X\subseteq V$, let $E_{G}[X]$ be the set of hyperedges $e$ of $G$ with $\supp (e)\subseteq X$ and $G[X]$ be the subgraph of $G$ induced by $X$, i.e., $G[X]=(X,E_{G}[X])$. For $\Acute{E}\subseteq E$, we define the closed neighbour set $N_{G}(\Acute{E})$ as follows $N_{G}(\Acute{E})=\{ e-u+v: e\in \Acute{E}, u\in e, v\in [n]\}\cap E$.

\begin{definition}
 A polynomial map $g:\mathbb({F}^d)^k\longrightarrow \mathbb{F}$ is said to be a multilinear $k$-form on $\mathbb{F}^d$ if it is linear on each argument on the vector space $\mathbb{F}^d$.
\end{definition}
\begin{theorem}
 Consider $g:(\mathbb{F}^d)^k\longrightarrow\mathbb{F}$ written as the sum of $t$  copies of a non-zero multilinear $k$-form $h:(\mathbb{F}^s)^k\longrightarrow\mathbb{F}$, where $st=d$. Let $G=([n],E)$ be a $k$-uniform hypergraph on $[n]$ and $\{ X_{1},...,X_{t}\}$ be a family of subsets $X_{i}$ of $[n]$ such that $|X_{i}|\ge n_{\Gamma _{h}}$. Denote $F_{i}=N_{G}(E_{G}[X_{i}])$ for each $i=1,...,t$. Suppose that:

 (P1) for every $i$, the hypergraph $([n],F_{i})$ is locally $h$-rigid as a hypergraph on $[n]$;

 (P2) for every $i$, $G[X_{i}]$ is locally $h$-rigid;

 (P3) for any $i, j$ with $i\neq j$, and for any $e\in F_{i}$ and $v\in e$, $\supp (e-v)\not\subseteq X_{j}$.
 Then $G$ is locally $g$-rigid.
\end{theorem}

\subsection{Sufficient Conditions for Global $g$-rigidity}

In the previous section, we gave a sufficient condition for the local rigidity of hypergraphs, in this section we will extend this result to global $g$-rigidity over the field of complex numbers.

\begin{proposition}
If a hypergraph $G$ is globally $g$-rigid over $\mathbb{C}$, then it is globally $g$-rigid over $\mathbb{R}$. 
\end{proposition}
\begin{proof}
 Suppose that $G$ is globally $g$-rigid over $\mathbb{C}$, hence any generic hyperframework $(G,p)$ of $G$ over $\mathbb{C}$ is globally rigid. Any generic configuration $q$ over $\mathbb{R}$ is also generic over $\mathbb{C}$ because genericity is defined in terms of algebraic independence over $\mathbb{Q}$. The hyperframework $(G,q)$ is generic over $\mathbb{C}$ for any real configuration $q$, hence it is globally rigid.
\end{proof}

\begin{proposition}
Assume $g:(\mathbb{F}^d)^k\longrightarrow\mathbb{F}$ is the sum of $t$ copies of a homogeneous $h:(\mathbb{F}^s)^k\longrightarrow\mathbb{F}$. Suppose that $\overline{\im {f_{h,G}}}$ is $t$-identifiable. Then $G$ is globally $g$-rigid if and only if $G$ is globally $h$-rigid.
\end{proposition}
\begin{proof}
 Global $h$-rigidity is necessary for global $g$-rigidity. To prove the converse direction of the proposition, let us take any generic $p\in (\mathbb{F}^d)^V$. Since $g$ is the sum of $t$ copies of $h$, $p\in (\mathbb{F}^d)^V$ can be decomposed into $(q_{1},...,q_{t})$ such that $f_{g,G}(p)=\sum_{i=1}^{t} f_{h,G}(q_{i})$, where $q_{i}\in (\mathbb{F}^s)^V$. Each $q_{i}$ is generic because of genericity of $p$. So $(G,q_{i})$ is globally $h$-rigid and $t$-identifiability of $\overline{\im {f_{h,G}}}$ implies the global $g$-rigidity of $(G,p)$.
\end{proof}

Let $V$ be a variety over $\mathbb{C}$ and $T_{x}V$ be a tangent space at $x\in V$.
\begin{definition}
 The $t$-tangential contact locus $C$ at generic $t$ points $x_{1},...,x_{t}$ is defined as follows: $C=\overline{\{ y\in V:T_{y}V\subseteq \langle T_{x_{1}}V,...,T_{x_{t}}V\rangle}$. A variety $V$ is called $t$-tangentially weakly defective if $C$ is not one dimensional. 
\end{definition}

\begin{theorem}[Chiantini and Ottaviani]
Assume that an affine variety $V$ over $\mathbb{C}$ is $t$-tangentially weakly non-defective. Then $V$ is $t$-identifiable.
\end{theorem}

\begin{corollary}
 Suppose a homogeneous $g:(\mathbb{C}^d)^k\longrightarrow\mathbb{C}$ is the sum of $t$ copies of $h:(\mathbb{C}^s)^k\longrightarrow\mathbb{C}$. Then $G$ is globally $g$-rigid if
\begin{itemize}
    \item 
    $\overline{\im f_{h,G}}$ is $t$-tangentially weakly non-defective
    \item $G$ is globally $h$-rigid.
\end{itemize}
 
\end{corollary}

\begin{proof}
 A combination of Proposition 4.4 and Theorem 4.6.   
\end{proof}

\begin{remark}The $t$-tangentially weakly defectiveness of $\overline{\im f_{h,G}}$ can be checked by characterizing the $t$-tangential contact locus in algebraic way.  Let $G$ be a $k$-uniform hypergraph with $n$ vertices and $m$ edges, and $g:(\mathbb{C}^d)^k\longrightarrow\mathbb{C}$ is the sum of $t$ copies of $h:(\mathbb{C}^s)^k\longrightarrow\mathbb{C}$. Let $p\in (\mathbb{C}^d)^n$ be a generic point-configuration, and let $\omega_{1},...,\omega_{k}\in \mathbb{C}^m$ be a basis of the orthogonal complement of the tangent space of $\overline{\im f_{g,G}}$ at $f_{g,G}(p)$. Since the tangent space of $\overline{\im f_{h,G}}$ at $f_{h,G}(q)$ is the image of $\jac f_{h,G}(q)$, the $t$-tangential contact locus $C$ is the closure of $\{q\in (\mathbb{C}^s)^n: \omega_{i}\in \ker \jac f_{h,G}(q)^{\top} (i=1,...,k)\}$. $\omega_{i}\in \ker \jac f_{h,G}(q)^{\top}$
forms a polynomial system in $q$, hence there is a polynomial map $\alpha_{i}:(\mathbb{C}^s)^n\longrightarrow (\mathbb{C}^s)^n$ such that $\omega_{i}\in \ker \jac f_{h,G}(q)^{\top}$ if and only if $\alpha_{i}(q)=0$. Therefore, the tangential contact locus is written as the closure of $\{ q\in (\mathbb{C}^s)^n: \alpha_{i}(q)=0 
 (i=1,...,k)\}$ and in order to check the $t$-tangentially weakly defectiveness of $\overline{\im f_{h,G}}$ it is enough to examine the rank of the Jacobian matrix of $\alpha_{i}$.
 \end{remark}
 
\begin{remark}
 If the map $g$ is the squared Euclidean distance, then the resulting Jacobian matrix of $\alpha_{i}$ is the Laplacian matrix of $G$ weighted by $\omega_{i}$. 
\end{remark}

\begin{definition}[Laplacian Matrix]
 Let $G = (V,E)$ be a simple graph. The Laplacian matrix $G$ is the matrix whose entries are defined as follows:
 $$L_{ij}=\begin{cases}
   \deg(v_{i}), &\hbox{if $i=j$}\\
   -1, &\hbox{$i\neq j$ and $v_{i}v_{j}\in E$}\\
   0, &\hbox{otherwise}
 \end{cases}$$
\end{definition}

\begin{example}
Let $G$ be the graph $(V = \{1,2,3,4\}, E = \{12, 14, 23, 34\})$. We want to compare its the resulting Jacobian matrix of $\alpha_{i}$ with its Laplacian matrix. 

Let $p=\begin{pmatrix}
    x_{a_1}&x_{b_1}&x_{c_1}&x_{d_1}\\
    x_{a_2}&x_{b_2}&x_{c_2}&x_{d_2}
\end{pmatrix}$
be a point-configuration which sends vertices 1,2,3 and 4 to $(x_{a_1},x_{a_2}),(x_{b_1},x_{b_2}),(x_{c_1},x_{c_2})$ and $(x_{d_1},x_{d_2})$, respectively. In this example, the map $g$ is the squared Euclidean distance, therefore $g$ is the sum of 2 copies of $h:\mathbb{R}^2\longrightarrow\mathbb{R}$, where $h(x,y)=(x-y)^2$. Then, by definition $h$-measurement map of $G$ is defined 

$f_{h,G}(x_{a_{1}},x_{b_{1}},x_{c_{1}},x_{d_{1}})=((x_{a_{1}}-x_{b_{1}})^2, (x_{b_{1}}-x_{c_{1}})^2, (x_{c_{1}}-x_{d_{1}})^2, (x_{d_{1}}-x_{a_{1}})^2)$.

Its Jacobian matrix equals 

\begin{center}
$\jac _{f_{h,G}}=
\begin{pmatrix}
 2(x_{a_1}-x_{b_1})&0&0&-2(x_{d_{1}}-x_{a_{1}})\\
 -2(x_{a_{1}}-x_{b_{1}})&2(x_{b_{1}}-x_{c_{1}})&0&0\\
 0&-2(x_{b_{1}}-x_{c_{1}})&2(x_{c_{1}}-x_{d_{1}})&0\\
 0&0&-2(x_{c_{1}}-x_{d_{1}})&2(x_{d_{1}}-x_{a_{1}})
\end{pmatrix}$
\end{center}
Let find left kernel of $\jac {f_{h,G}}(p)^{\top}$. The vector $v=(v_{1},v_{2},v_{3},v_{4})\in \jac {f_{h,G}}(p)^{\top}$ if and only if its coordinates satisfy the system of equations:
$$\begin{cases}
  2(x_{a_{1}}-x_{b_{1}})v_{1}-2(x_{d_{1}}-x_{a_{1}})v_{4}=0\\

  -2(x_{a_{1}}-x_{b_{1}})v_{1}+2(x_{b_{1}}-x_{c_{1}})v_{2}=0\\

  -2(x_{b_{1}}-x_{c_{1}})v_{2}+2(x_{c_{1}}-x_{d_{1}})v_{3}=0\\

  -2(x_{c_{1}}-x_{d_{1}})v_{3}+2(x_{d_{1}}-x_{a_{1}})v_{4}=0  
\end{cases} \hbox{or equivalently,} \quad \begin{cases}
(2v_{1}+2v_{4})x_{a_{1}}-2v_{1}x_{b_{1}}-2v_{4}x_{d_{1}}=0\\

-2v_{1}x_{a_{1}}+(2v_{1}+2v_{2})x_{b_{1}}-2v_{2}x_{c_{1}}=0\\

-2v_{2}x_{b_{1}}+(2v_{2}+2v_{3})x_{c_{1}}-2v_{3}x_{d_{1}}=0\\

-2v_{4}x_{a_{1}}-2v_{3}x_{c_{1}}+(2v_{3}+2v_{4})x_{d_{1}}=0
\end{cases}$$
Then, the resulting Jacobian matrix of polynomials $\alpha_{i}$ equals
$$\begin{pmatrix}
 2v_{1}+2v_{4}&-2v_{1}&0&-2v_{4}\\
 -2v_{1}&2v_{1}+2v_{2}&-2v_{1}&0\\
 0&-2v_{2}&2v_{2}+2v_{3}&-2v_{3}\\
 -2v_{4}&0&-2v_{3}&2v_{3}+2v_{4}
\end{pmatrix}$$
On the other hand, its Laplacian matrix equals: 

$$\begin{pmatrix}
 2&-1&0&-1\\
 -1&2&-1&0\\
 0&-1&2&-1\\
 -1&0&-1&2
\end{pmatrix}$$

Indeed, the Jacobian matrix of $\alpha_{i}$ is the Laplacian matrix of $G$ weighted by $v_{i}$.

\end{example}

\begin{example}
 Consider the symmetric matrix completion problem with symmetric rank $d=3$, order $k=2$ and size $n=3$. Let $$T=\begin{pmatrix}
     x_{11}&0&0\\
     0&x_{22}&0\\
     0&0&x_{33}
 \end{pmatrix}$$  
 and $g$ is the Euclidean inner product. We want to compare its the resulting Jacobian matrix of $\alpha_{i}$ with its adjacency matrix.

The matrix $T$ can be written in the form: 

\begin{align*}
    T &= x_{1}x^{\top}_{1}+x_{2}x^{\top}_{2}+x_{3}x^{\top}_{3}=\\
    \hfill &= \begin{pmatrix}
     y_{11}\\
     y_{12}\\
     y_{13}
 \end{pmatrix} \cdot \begin{pmatrix}
   y_{11}&y_{12}&y_{13}  
 \end{pmatrix} + \begin{pmatrix}
     y_{21}\\
     y_{22}\\
     y_{23}
 \end{pmatrix} \cdot \begin{pmatrix}
   y_{21}&y_{22}&y_{23}  
 \end{pmatrix}+\begin{pmatrix}
     y_{31}\\
     y_{32}\\
     y_{33}
 \end{pmatrix} \cdot \begin{pmatrix}
   y_{31}&y_{32}&y_{33}  
 \end{pmatrix}
\end{align*}

 Let $p$ be a $3\times 3$ matrix obtained by aligning $x_{i}$ as the $i$-th row vector, $x_{i}=
 \begin{pmatrix}
  y_{i1}&
  y_{i2}&
  y_{i3}
 \end{pmatrix}^{\top}$,
 
 $p=\begin{pmatrix}
     y_{11}&y_{12}&y_{13}\\
     y_{21}&y_{22}&y_{23}\\
     y_{31}&y_{32}&y_{33}
 \end{pmatrix}$. Then, we have 
 $$\begin{cases}
     p(1)p(2)=0\\
     p(1)p(3)=0\\
     p(2)p(3)=0
 \end{cases}$$
 where $p(j)$ is the $j$-th column of the matrix $p$.

 The $g$-measurement map of $G$ is defined $f_{g,G}(p)=(p(1)p(2),p(1)p(3),p(2)p(3))$. The map $g$ is the sum of 3 copies of $h$, hence $f_{h,G}(y_{11},y_{12},y_{13})=(y_{11}y_{12},y_{11}y_{13},y_{12}y_{13})$.
\begin{center}

 $\jac f_{h,G}(p)=\begin{pmatrix}
     y_{12}&y_{11}&0\\
     y_{13}&0&y_{11}\\
     0&y_{13}&y_{12}
 \end{pmatrix}$
 \end{center}
 \begin{center}

 $\begin{pmatrix}
  v_{1}&v_{2}&v_{3}   
 \end{pmatrix}$
 $\begin{pmatrix}
     y_{12}&y_{11}&0\\
     y_{13}&0&y_{11}\\
     0&y_{13}&y_{12}
 \end{pmatrix}$=0
 \end{center}

 $$\begin{cases}
  v_{1}y_{12}+v_{2}y_{13}=0\\
  v_{1}y_{11}+v_{3}y_{13}=0\\
  v_{2}y_{11}+v_{3}y_{12}=0
 \end{cases}$$

 The resulting Jacobian matrix of $\alpha_{i}$ equals
 $\begin{pmatrix}
   0&v_{1}&v_{2}\\
   v_{1}&0&v_{3}\\
   v_{2}&v_{3}&0
 \end{pmatrix}$. On the other hand, if we weight edges $\{\{12\},\{13\},\{23\}\}$ with vector $v=(v_{1},v_{2},v_{3})$, then its adjacency matrix equals $\begin{pmatrix}
   0&v_{1}&v_{2}\\
   v_{1}&0&v_{3}\\
   v_{2}&v_{3}&0
 \end{pmatrix}$ which corresponds with the Jacobian matrix.  
\end{example}

\begin{example}
 Consider the symmetric tensor completion problem of symmetric rank one and order three, that is $d=1$, $k=3$ and $g=h_{prod}$. $G$ is 3-uniform hyper-framework: $G=\{\{a,b,c\}, \{aaa,aab,abc\}\}$ and $p:a\longmapsto x_{a}, b\longmapsto x_{b}$ and $c\longmapsto x_{c}$. We want to analyze its the resulting Jacobian matrix of $\alpha_{i}$ and its adjacency matrix.

 The map $g$ is given by $g:\mathbb{C}^3\longrightarrow \mathbb{C}$ such that it sends the tuple $(x_{a},x_{b},x_{c})$ to the tuple $(x^3_{a},x^2_{a}x_{b},x_{a}x_{b}x_{c})$. The map $g$ equals to the map $h$, hence 
 \begin{center}

 $\jac f_{h,G}=\begin{pmatrix}
   3x^2_{a}&0&0\\
   2x_{a}x_{b}&x^2_{a}&0\\
   x_{b}x_{c}&x_{a}x_{c}&x_{a}x_{b}
 \end{pmatrix}$

$\begin{pmatrix}
  v_{1}&v_{2}&v_{3}  
\end{pmatrix}
\begin{pmatrix}
 3x^2_{a}&0&0\\
   2x_{a}x_{b}&x^2_{a}&0\\
   x_{b}x_{c}&x_{a}x_{c}&x_{a}x_{b}   
\end{pmatrix}=0$

 $\begin{cases}
  3x^2_{a}v_{1}+2x_{a}x_{b}v_{2}+x_{b}x_{c}v_{3}=0\\
  x^2_{a}v_{2}+x_{a}x_{c}v_{3}=0\\
  x_{a}x_{b}v_{3}=0
 \end{cases}$
\end{center}

The Jacobian matrix of $\alpha_{i}$ polynomial can be written as the multiplication of the following two matrices:
\begin{center}
$\begin{pmatrix}
 6v_{1}x_{a}+2v_{2}x_{b}&2v_{2}x_{c}+v_{3}x_{c}&v_{3}x_{b}\\
 2v_{2}x_{a}+v_{3}x_{c}&0&v_{3}x_{c}\\
 v_{3}x_{b}&v_{3}x_{a}&0
\end{pmatrix}
=\begin{pmatrix}
3v_{1}&2v_{2}&0&v_{3}\\
v_{2}&0&v_{3}&0\\
0&v_{3}&0&0
\end{pmatrix}
\begin{pmatrix}
 2x_{a}&0&0\\
 x_{b}&x_{c}&0\\
 x_{c}&0&x_{c}\\
 0&x_{c}&x_{b}
\end{pmatrix}$
\end{center}
Interesting point is that the first matrix of the RHS, 
$\begin{pmatrix}
3v_{1}&2v_{2}&0&v_{3}\\
v_{2}&0&v_{3}&0\\
0&v_{3}&0&0
\end{pmatrix}$ is the adjacency matrix of the hypergraph.

\end{example}

\begin{example}\label{exam:Emiliano}
 Let $G=(V,E)$ be a 3-uniform hyper-graph with $V=\{a,b,c,d\}$ and $E=\{aaa,aab,abc,bcd\}$. Let the symmetric rank equals 2, that is $d=2$ and $p$ be a point-configuration sends $a\longmapsto (x_{a_{1}},x_{a_{2}})$,
 $b\longmapsto (x_{b_{1}},x_{b_{2}})$,
 $c\longmapsto (x_{c_{1}},x_{c_{2}})$ and
 $d\longmapsto (x_{d_{1}},x_{d_{2}})$.  We want to analyze the relation between its the resulting Jacobian matrix of $\alpha_{i}$ and its adjacency matrix.

 The $g$-measurement map of $G$ is defined: $f_{g,G}(p)=(x^3_{a_{1}}+x^3_{a_{2}},x^2_{a_{1}}x_{b_{1}}+x^2_{a_{2}}x_{b_{2}},x_{a_{1}}x_{b_{1}}x_{c_{1}}+
 +x_{a_{2}}x_{b_{2}}x_{c_{2}},x_{b_{1}}x_{c_{1}}x_{d_{1}}+x_{b_{2}}x_{c_{2}}x_{d_{2}})$. In this example, $g$ is the sum of two copies of $h$ and hence, $h$-measurement map of $G$ is defined: $f_{h,G}(p)=(x^3_{a_{1}},x^2_{a_{1}}x_{b_{1}},x_{a_{1}}x_{b_{1}}x_{c_{1}},x_{b_{1}}x_{c_{1}}x_{d_{1}})$.
\begin{center}
 $\jac f_{h,G}=\begin{pmatrix}
  3x^2_{a_{1}}&0&0&0\\
  2x_{a_{1}}x_{b_{1}}&x^2_{a_{1}}&0&0\\
  x_{b_{1}}x_{c_{1}}&x_{a_{1}}x_{c_{1}}&x_{a_{1}}x_{b_{1}}&0\\
  0&x_{c_{1}}x_{d_{1}}&x_{b_{1}}x_{d_{1}}&x_{b_{1}}x_{c_{1}}
 \end{pmatrix}$
 \end{center}
 Let find the left kernel of the Jacobian matrix and corresponding Jacobian matrix of $\alpha_{i}$ polynomials.
 \begin{center}
 $\begin{pmatrix}
   v_{1}&v_{2}&v_{3}&v_{4}  
 \end{pmatrix}$
 $\begin{pmatrix}
  3x^2_{a_{1}}&0&0&0\\
  2x_{a_{1}}x_{b_{1}}&x^2_{a_{1}}&0&0\\
  x_{b_{1}}x_{c_{1}}&x_{a_{1}}x_{c_{1}}&x_{a_{1}}x_{b_{1}}&0\\
  0&x_{c_{1}}x_{d_{1}}&x_{b_{1}}x_{d_{1}}&x_{b_{1}}x_{c_{1}}
 \end{pmatrix}$=0
 
$\begin{cases}
 3v_{1}x^2_{a_{1}}+2v_{2}x_{a_{1}}x_{b_{1}}+v_{3}x_{b_{1}}x_{c_{1}}=0\\
 v_{2}x^2_{a_{1}}+v_{3}x_{a_{1}}x_{c_{1}}+v_{4}x_{c_{1}}x_{d_{1}}=0\\
 v_{3}x_{a_{1}}x_{b_{1}}+v_{4}x_{b_{1}}x_{d_{1}}=0\\
 v_{4}x_{b_{1}}x_{c_{1}}=0
\end{cases}$

\begin{multline*}
\begin{pmatrix}
 6v_{1}x_{a_{1}}+2v_{2}x_{b_{1}}&2v_{2}x_{a_{1}}+v_{3}x_{c_{1}}&v_{3}x_{b_{1}}&0\\
 2v_{2}x_{a_{1}}+v_{3}x_{c_{1}}&0&v_{3}x_{a_{1}}+v_{4}x_{d_{1}}&v_{4}x_{c_{1}}\\
 v_{3}x_{b_{1}}&v_{3}x_{a_{1}}+v_{4}x_{d_{1}}&0&v_{4}x_{b_{1}}\\
 0&v_{4}x_{c_{1}}&v_{4}x_{b_{1}}&0
\end{pmatrix}= \\
= \begin{pmatrix}
 3v_{1}&2v_{2}&0&v_{3}&0&0\\
 v_{2}&0&v_{3}&0&0&v_{4}\\
 0&v_{3}&0&0&v_{4}&0\\
 0&0&0&v_{4}&0&0
\end{pmatrix}
\begin{pmatrix}
 2x_{a_{1}}&0&0&0\\
 x_{b_{1}}&x_{a_{1}}&0&0\\
 x_{c_{1}}&0&x_{a_{1}}&0\\
 0&x_{c_{1}}&x_{b_{1}}&0\\
 0&x_{d_{1}}&0&x_{b_{1}}\\
 0&0&x_{d_{1}}&x_{c_{1}}
\end{pmatrix}\end{multline*}

 \end{center}
 The first matrix of RHS corresponds to the adjacency matrix of the hyper-graph $G$ weighted by the vector $(v_{1},v_{2},v_{3},v_{4})$. The second matrix has the full rank, hence finding the rank of the resulting Jacobian matrix of polynomials $\alpha_{i}$ can be reduced finding the rank of adjacency matrix of $G$.
\end{example}

\begin{theorem}
 Let $V$ be a non-degenerate affine variety in $\mathbb{C}^m$. Suppose that $(t+1)\dim V\le m$ and $V$ is $(t+1)$-non-defective and 1-tangentially weakly non-defective. Then $V$ is $t$-identifiable.
\end{theorem}

\begin{corollary}
 For a homogeneous $h:(\mathbb{C}^s)^k\longrightarrow\mathbb{C}$ and a positive integer $d$, let $dh$ be the sum of $d$ copies of $h$, and suppose $d_{\Gamma_{(d+1)h}}=(d+1)d_{\Gamma_{h}}$. Then, for a positive integer $d$, $G$ is globally $dh$-rigid if
 \begin{itemize}
     \item $G$ is locally $(d+1)h$-rigid,
     \item $\overline{\im f_{h,G}}$ is 1-tangentially weakly non-defective, and
     \item $G$ is globally $h$-rigid.
 \end{itemize}
\end{corollary}
\begin{proof}
 A combination of Proposition 4.4 and Theorem 4.14.   
\end{proof}

Let define adjacency matrix of an edge weighted hypergraph. Let $G=(V,E)$ be a $k$-uniform hypergraph and $w\in \mathbb{F}^E$ is a vector representing the weight of each edge. Consider the collection $E^{(k-1)}:=\{ \sigma \in \binom{V}{k-1}:\sigma\subseteq e\in E\}$ of multisets of size $(k-1)$ contained in some hyperedge in $G$. The adjacency matrix $A_{G,w}$ is a $\mathbb{F}$-matrix of size $|V|\times |E^{(k-1)}|$ such that:
$$A_{G,w}[v,\sigma]=\begin{cases}
   m_{e}(v)w(e) &\text{(if $e=\sigma+v$ is in $G$)} \\
   0 &\text{(otherwise).} \end{cases}$$
where each row is indexed by a vertex $v\in V$ and each column is indexed by $\sigma\in E^{(k-1)}$.
Similarly, in the case of anti-symmetric functions, we consider the signed variant of the adjacency matrix:
$$A^{s}_{G,w}[v,\sigma] =\begin{cases}
   \sign (e,v)m_{e}(v)w(e) &\text{(if $e=\sigma+v$ is in $G$)} \\
   0 &\text{(otherwise).}
 \end{cases} $$
for each $v\in V$ and $\sigma \in E^{(k-1)}$, where $\sign (e,v)$ denotes the standard sign function of permutations, which is positive (resp. negative) if the ordering of $v$ in $e$ is odd (resp.even).

\begin{theorem}
    Suppose $g \colon (\FF^d)^k \to \FF$ is a multilinear $k-$form and $(G,p)$ is a generic locally $g$-rigid $k$-uniform hyperframework. Suppose that $f$ is symmetric and $\dim \left(\bigcap_{\omega \in \ker \jac f_{G,g}(p)^{\top}} \ker A_{G,\omega}\right) = d$. Then for each $q \in f^{-1}_{g,G}(f_{g,G}(p))$ there exists $T v\in \FF^{d \times d}$ such that 
    \[\nabla g(q(\sigma)) = T \nabla g (p(f)) \qquad \hbox{ for all $\sigma \in E^{(k-1)}$}\]
\end{theorem}


\begin{proposition}
    Let $h = h_{prod} \colon \CC^k \to \CC$ be the product map, $G = (V,E)$ a hypergraph and $x_1, \dots, x_t$ some generic points in $\overline{\im f_{h,G}}$. If $t \leq |E| - \rank \jac f_{h,G}$, then any point in the intersection of $\overline{\im f_{h,G}}$ and $\sspan_{\CC} \{x_1, \dots, x_t\}$ is a scalar multiple $x_i$ for some $i = 1, \dots, t$.
\end{proposition}

\begin{theorem}
 Let $g$ is the sum of $d$ copies of $h_{prod}:\mathbb{F}^k\longrightarrow\mathbb{F}$. Then $G=(V,E)$ is globally $g$-rigid if there exists a point-configuration $p\in (\mathbb{F}^d)^V\longrightarrow\mathbb{F}$ such that
 \begin{itemize}
     \item $(G,p)$ is infinitesimally $g$-rigid,
     \item $|E^{(k-1)}| \ge |V|+d$, and
     \item $\dim (\cap_{w\in \ker Jf_{g,G}(p)} \ker A_{G,w})=d.$
 \end{itemize}
\end{theorem}

\begin{example}

Consider the symmetric tensor completion problem with symmetric rank one and order four: $k=4$ and $g=h_{prod}$. 4-uniform hyper-framework ($G$,$p$) is given by $G=(\{a,b\}, \{aaaa, aaab, bbbb\})$ and $p: a\rightarrow x_{a}, b\rightarrow x_{b}$. Is $G$ globally $g$-rigid?

Let find $g$-measurement map of $G$, by definition it is defined as follows:

 $f_{g,G}=(g(p(a),p(a),p(a),p(a)), g(p(a),p(a),p(a),p(b)), g(p(b),p(b),p(b),p(b)))=$

 $=(g(x_{a},x_{a},x_{a},x_{a}),g(x_{a},x_{a},x_{a},x_{b}),g(x_{b},x_{b},x_{b},x_{b})=(x^4_{a},x^3_{a}x_{b},x^4_{b})$

Then, its Jacobian matrix equals
\begin{center}

 $\jac f_{g,G}(p)=\begin{pmatrix}
                4x^3_{a} & 0 \\
                3x^2_{a}x_{b} & x^3_{a} \\
                0 & 4x^3_{b}
              \end{pmatrix}$
\end{center}
  Obviously, $\rank \jac f_{g,G}(p)=2$ which implies that $(G,p)$ is infinitesimally $g$-rigid. In the next step, let find the weighted adjacency matrice of the hypergraph.
\begin{center}
 $ \begin{pmatrix}
     A & B & C
   \end{pmatrix}$
 $ \begin{pmatrix}
                4x^3_{a} & 0 \\
                3x^2_{a}x_{b} & x^3_{a} \\
                0 & 4x^3_{b}
              \end{pmatrix}=0$

  $\begin{cases}
     4x^3_{a}A+3x^2_{a}x_{b}B=0 \\
     x^3_{a}B+4x^3_{b}C=0
   \end{cases} $
\end{center}
   W.L.O.G let take $B=1$, by solving system of equations we find that $A=-\frac{3x_{b}}{4x_{a}}$ and $C=-\frac{x^3_{a}}{4x^3_{b}}$. Hence, $\omega(aaaa)=-\frac{3x_{b}}{4x_{a}}$, $\omega(aaab)=1$ and $\omega(bbbb)=-\frac{x^3_{a}}{4x^3_{b}}$.
\begin{center}
 $A_{G,\omega}=\begin{pmatrix}
                 -\frac{3x_{b}}{x_{a}} & 3 & 0 \\
                 1 & 0 & -\frac{x^3_{a}}{x^3_{b}}
               \end{pmatrix}$
\end{center}
  Rank of this matrix equals to 1. It implies that $\dim \ker A=1$ and by applying Theorem 4.9 we conclude that $G$ is globally $g$-rigid.
\end{example}

\begin{example}
 Consider the symmetric tensor completion problem with symmetric rank one and order four: $k=4$ and $g=h_{prod}$. 4-uniform hyper-framework ($G$,$p$) is given by $G=(\{a,b\}, \{aaaa, abbb, aaab\})$ and $p: a\rightarrow x_{a}, b\rightarrow x_{b}$. Is $G$ globally $g$-rigid?

Similarly as above, let find $g$-measurement map of $G$:

 $f_{g,G_{1}}=(g(p(a),p(a),p(a),p(a)), g(p(a),p(b),p(b),p(b)), g(p(a),p(a),p(a),p(b)))=$

 $=(g(x_{a},x_{a},x_{a},x_{a}),g(x_{a},x_{b},x_{b},x_{b}),g(x_{a},x_{a},x_{a},x_{b})=(x^4_{a},x_{a}x^3_{b},x^3_{a}x_{b})$

 Then, its Jacobian matrix equals
\begin{center}
 $\jac f_{g,G}(p)=\begin{pmatrix}
                4x^3_{a} & 0 \\
                x^3_{b} & 3x_{a}x^2_{b} \\
                3x^2_{a}x_{b} & x^3_{a}
              \end{pmatrix}$
\end{center}
 It is clear that $\rank \jac f_{g,G}(p)=2$ which implies that $(G,p)$ is infinitesimally $g$-rigid. Let find the weighted adjacency matrix of the hypergraph.
\begin{center}
 $ \begin{pmatrix}
     A & B & C
   \end{pmatrix}$
 $ \begin{pmatrix}
                4x^3_{a} & 0 \\
                x^3_{b} & 3x_{a}x^2_{b} \\
                3x^2_{a}x_{b} & x^3_{a}
              \end{pmatrix}=0$

  $\begin{cases}
     4x^3_{a}A+x^3_{b}B+3x^2_{a}x_{b}C=0 \\
     3x_{a}x^2_{b}B+x^3_{a}C=0
   \end{cases} $
\end{center}
   W.L.O.G let take $B=1$, by solving system of equations we find that $A=\frac{2x^3_{b}}{x^3_{a}}$ and $C=-\frac{3x^2_{b}}{x^2_{a}}$. Hence, $\omega(aaaa)=\frac{2x^3_{b}}{x^3_{a}}, \omega(abbb)=1$ and $\omega(aaab)=-\frac{3x^2_{b}}{x^2_{a}}$.
\begin{center}
 $A_{G,\omega}=\begin{pmatrix}
                 \frac{6x^3_{b}}{x^3_{a}} & 0 & 1 & -\frac{9x^2_{b}}{x^2_{a}} \\
                 -\frac{3x^2_{b}}{x^2_{a}} & 3 & 0 & 0
               \end{pmatrix}$
\end{center}
  Rank of this matrix equals to 2. It implies that $\dim \ker A=2$ and hence we can not apply Theorem 4.9.
   
\end{example}
\section{Acknowledgements}
I would like to thank prof. Fatemeh Mohammadi and Giacomo Masiero for their supervision during my internship and also Emiliano Liwski for his contribution to Example~\ref{exam:Emiliano}.  

\section{References}

1. James Cruickshank, Fatemeh Mohammadi, Anthony Nixon and Shin-ichi Tanigawa. Identifiability of points and rigidity of hypergraphs under algebraic constraints (2023). ArXiv preprint arXiv:2305.18990

\medskip
\noindent
2. James Cruickshank, Fatemeh Mohammadi, Harshit J. Motwani, Anthony Nixon, and Shin-ichi Tanigawa. Global rigidity of line constrained frameworks (2022). arXiv preprint arXiv:2208.09308. SIAM Journal on Discrete Mathematics. 2024 Mar 31;38(1):743-63.

\medskip
\noindent 3. Fatemeh Mohammadi. Tensors in statistics and rigidity theory. arXiv preprint arXiv:2212.14752 (2022).

\end{document}